\documentclass[12pt,a4paper,reqno]{amsart}
\usepackage[T1]{fontenc}
\usepackage{lmodern}
\usepackage{amsmath,amssymb,mathtools,mathrsfs}
\usepackage{microtype}
\usepackage[a4paper,hmargin=28mm,top=26mm,bottom=28mm,includeheadfoot]{geometry}
\usepackage{tikz-cd}
\usepackage{booktabs}
\usepackage{xurl}
\usepackage{xcolor}
\definecolor{linkblue}{RGB}{35,64,105}
\usepackage[colorlinks=true,linkcolor=linkblue,citecolor=linkblue,urlcolor=linkblue]{hyperref}

\allowdisplaybreaks[1]

\newtheoremstyle{readableplain}
  {9pt plus 2pt minus 1pt}{8pt plus 2pt minus 1pt}
  {\itshape}{}{\bfseries}{.}{.5em}{}
\newtheoremstyle{readabledefinition}
  {9pt plus 2pt minus 1pt}{8pt plus 2pt minus 1pt}
  {\normalfont}{}{\bfseries}{.}{.5em}{}
\newtheoremstyle{readableremark}
  {9pt plus 2pt minus 1pt}{8pt plus 2pt minus 1pt}
  {\normalfont}{}{\itshape}{.}{.5em}{}
\theoremstyle{readableplain}
\newtheorem{theorem}{Theorem}[section]
\newtheorem{proposition}[theorem]{Proposition}
\newtheorem{lemma}[theorem]{Lemma}
\newtheorem{corollary}[theorem]{Corollary}

\newtheorem{conjecture}{Conjecture}[section]
\theoremstyle{readabledefinition}

\theoremstyle{readableremark}
\newtheorem{remark}[theorem]{Remark}

\makeatletter
\renewcommand{\section}{\@startsection{section}{1}{\z@}
  {-2.5ex\@plus-.8ex\@minus-.2ex}{1.2ex\@plus.2ex}
  {\normalfont\large\bfseries\raggedright}}
\renewcommand{\subsection}{\@startsection{subsection}{2}{\z@}
  {-2ex\@plus-.6ex\@minus-.2ex}{.8ex\@plus.2ex}
  {\normalfont\bfseries\raggedright}}
\renewcommand{\subsubsection}{\@startsection{subsubsection}{3}{\z@}
  {-1.5ex\@plus-.4ex\@minus-.2ex}{.6ex\@plus.1ex}
  {\normalfont\bfseries\raggedright}}
\makeatother

\numberwithin{equation}{section}

\DeclareMathOperator{\tr}{tr}

\newcommand{\PP}{\mathbb P}

\newcommand{\dbar}{\bar\partial}
\newcommand{\dd}{\mathrm d}

\DeclareMathOperator{\Ric}{Ric}

\newcommand{\RR}{\mathscr R}
\newcommand{\dV}{\mathrm dV}
\newcommand{\ii}{\sqrt{-1}}

\title[Positivity of the canonical bundle]{Positivity of the Canonical Bundle and Hermitian Metrics with Quasi-Negative Holomorphic Sectional Curvature}
\author{Xueyuan Wan}
\address{Xueyuan Wan: Mathematical Science Research Center, Chongqing University of Technology,
Chongqing 400054, China}
\email{xwan@cqut.edu.cn}
\makeatletter
\@namedef{subjclassname@2020}{\textup{2020} Mathematics Subject Classification}
\makeatother
\date{}
\hypersetup{
  pdftitle={Positivity of the Canonical Bundle and Hermitian Metrics with Quasi-Negative Holomorphic Sectional Curvature},
  pdfauthor={Xueyuan Wan},
  pdfsubject={Canonical positivity on compact Kahler manifolds},
  pdfkeywords={Canonical bundle, Hermitian metric, quasi-negative holomorphic sectional curvature, Chern--Lu identity}
}

\begin{document}
\begin{abstract}
We prove that if a compact K\"ahler manifold admits a Hermitian metric
with quasi-negative holomorphic sectional curvature for the Chern
connection, then its canonical bundle is ample. The Hermitian metric
realizing this curvature need not be K\"ahler. This proves the K\"ahler case of
the quasi-negative conjecture of Yang and Zheng.
\end{abstract}
\maketitle

\section{Introduction}\label{sec:introduction}

The canonical bundle links the differential geometry of a complex
manifold to its algebraic geometry. Its positivity governs
pluricanonical embeddings and is central to the classification of
complex manifolds. Negative curvature is a natural analytic source of
this positivity: a K\"ahler metric with negative Ricci curvature
induces a positive Hermitian metric on the canonical bundle, while the
Aubin--Yau theorem produces a K\"ahler--Einstein metric with negative
Ricci curvature whenever the canonical bundle is ample
\cite{Aubin78,YauMA78}. A basic question is whether negativity of a
weaker curvature quantity such as holomorphic sectional curvature
already forces ampleness.

For K\"ahler metrics, the answer is affirmative. Heier, Lu, and Wong
\cite{HLW10} proved ampleness for projective threefolds and obtained
higher-dimensional results. Wu and Yau \cite{WY16} proved the theorem
for projective manifolds in every dimension. Tosatti and Yang
\cite{TY17} removed the projectivity assumption: a compact K\"ahler
manifold with a K\"ahler metric of negative holomorphic sectional
curvature has ample canonical bundle. They also proved nefness under
nonpositive holomorphic sectional curvature. Diverio and Trapani
\cite{DT19} treated quasi-negative holomorphic sectional curvature in
the K\"ahler-metric case,
and Wu and Yau \cite{WY16b} gave a unified proof of these results. An
alternative approach via the K\"ahler--Ricci flow was developed by
Nomura \cite{Nomura18}.

The statements above involve two distinct K\"ahler conditions. One
concerns the complex manifold itself; the other concerns the metric
whose curvature is prescribed. A manifold may admit a K\"ahler metric
even when the Hermitian metric under consideration is not K\"ahler.
This distinction matters here because the Chern curvature of a general
Hermitian metric lacks the symmetries used in the K\"ahler Schwarz
lemma. 

The Hermitian question belongs to a broader program relating negative
curvature, hyperbolicity, and positivity of the canonical bundle; see
\cite{Kobayashi98}. Yang and Zheng formulated the following
quasi-negative conjecture~\cite[Conjecture~1.1(b)]{YZ19}.

\begin{conjecture}[{\cite[Conjecture~1.1(b)]{YZ19}}]\label{conj:YZ}
If a compact complex manifold admits a Hermitian metric with
quasi-negative holomorphic sectional curvature, then its canonical
bundle is ample.
\end{conjecture}

The original conjectures do not assume that the manifold is
K\"ahler. Here holomorphic sectional curvature is computed using the
Chern connection. In the conjecture, quasi-negative means that
$H_h\leq0$ everywhere and is strictly negative in every direction at
some point. We prove this statement when the underlying complex
manifold is K\"ahler; no K\"ahler condition is imposed on the
Hermitian metric $h$.

Several results indicate both the scope and the difficulty of this
extension. Yang and Zheng introduced real bisectional curvature and
proved nefness and ampleness of the canonical bundle under appropriate
nonpositive and quasi-negative real bisectional curvature assumptions
on a compact complex manifold that admits a K\"ahler metric
\cite[Theorems~1.7 and~1.8]{YZ19}.
Broder and Stanfield \cite{BS26} proved that the canonical bundle is
ample on a compact K\"ahler manifold carrying a pluriclosed Hermitian
metric with negative holomorphic sectional curvature. In a different
direction, Lee and Streets \cite{LS21} studied the negative Chern
curvature-operator condition. In complex dimension two, they showed that this
condition forces the canonical bundle to be ample, without assuming
that the underlying surface is K\"ahler.

More recently, Tang \cite[Theorem~1.1(1)]{Tang26} proved that a
compact K\"ahler manifold carrying an arbitrary Hermitian metric with
nonpositive holomorphic sectional curvature has nef canonical bundle.
Tang also proved ampleness under strictly negative holomorphic sectional
curvature in complex dimension two and, under a balancedness assumption
on the Hermitian metric, in complex dimension three. In their recent preprint, Qin and
Tian \cite[Lemmas~3.2 and~3.3]{QT26} established an integrated
Chern--Lu identity and used it to prove ampleness of the canonical
bundle in the constant negative Chern holomorphic sectional curvature
case. 

Our main result is the following.

\begin{theorem}\label{thm:main}
Let $M$ be a compact K\"ahler manifold. If $M$ admits a
Hermitian metric with quasi-negative Chern holomorphic sectional
curvature, then the canonical bundle $K_M$ is ample. In particular,
$M$ is projective.
\end{theorem}

The constant negative-curvature case is covered by
\cite[Lemma~3.3]{QT26}; the argument below treats the quasi-negative
case, where no global negative constant is available.
Theorem~\ref{thm:main} proves the K\"ahler case of the quasi-negative
Yang--Zheng conjecture. It assumes neither constant curvature nor a
pluriclosed or balanced condition on $h$. The metric $h$ itself need
not be K\"ahler; only the underlying manifold is assumed K\"ahler.

The key analytic idea is to use the symmetrized curvature contraction
rather than the contraction appearing directly in the Chern--Lu
formula. Let $g$ be an auxiliary K\"ahler metric and set
\[
 u=\tr_g h,
 \qquad
 q=\tr\bigl((g^{-1}h)^2\bigr).
\]
In a $g$-unitary frame, put
\[
 \RR_g(h)=\frac12\sum_{i,k}
 \bigl(R^h_{i\bar i k\bar k}+R^h_{i\bar k k\bar i}\bigr).
\]
If $H_h\leq0$ everywhere and $H_h(p,[\xi])<0$ for every complex line
$[\xi]\in\mathbb P(T^{1,0}_pM)$ at some point $p$, Berger's projective
averaging gives
\[
 \RR_g(h)\leq0,
\]
with strict inequality at $p$, and hence on a neighborhood of $p$ by
continuity. This averaging step uses only
the holomorphic sectional curvature of $h$; it does not require the
curvature tensor of $h$ to have the additional symmetries enjoyed by a
K\"ahler metric. 

For quasi-negative curvature, this identity is combined with an
$(n-1)$-Hessian auxiliary equation. The equation first produces an
admissible solution; the zero-direction argument shows that the
solution is in fact K\"ahler. The strict negativity on a fixed
neighborhood then gives a uniform mixed-volume estimate, which is the
main quantitative input for proving that $K_M$ is big.

Once $K_M$ is known to be nef, fix a K\"ahler form $\omega_0$ and write
\[
 \alpha=2\pi c_1(K_M),\qquad \beta=[\omega_0].
\]
Instead of a twisted K\"ahler--Einstein equation, we fix the positive
volume form $\Omega=\omega_0^n/n!$ and solve an $(n-1)$-Hessian
equation whose mixed volume form satisfies
\[
 \frac{\omega_h\wedge\omega_t^{n-1}}{(n-1)!}=e^{\varphi_t}\Omega.
\]
The resulting K\"ahler forms have class $\alpha+t\beta$. The
integrated Chern--Lu identity, together with the extra
$dd^c\log u_t$ term supplied by this equation, yields a uniform
estimate for the mixed volume on the open set of strict negativity.
The fixed K\"ahler form then turns this estimate into a differential
inequality for the cohomological volume polynomial
\[
 V(t)=\frac1{n!}\int_M(\alpha+t\beta)^n.
\]
This forces $V(0)>0$. The canonical bundle is therefore big. We then
use projectivity and the absence of rational curves to deduce
ampleness; the details appear below.

The paper is organized as follows. Section~\ref{sec:curvature} fixes
the curvature conventions and proves the projective averaging
inequality. Section~\ref{sec:integrated} establishes the integrated
Chern--Lu identity. Section~\ref{sec:canonical} constructs the
$(n-1)$-Hessian auxiliary metrics and proves Theorem~\ref{thm:main}.

\section{Curvature conventions and projective averaging}
\label{sec:curvature}

Throughout the paper, $M$ denotes a connected compact complex manifold
of complex dimension $n$. For a Hermitian metric $h$ on $T^{1,0}M$,
its fundamental form is
\[
 \omega_h=\ii h_{i\bar j}\,\dd z^i\wedge\dd\bar z^j.
\]
Repeated indices are summed unless a summation sign is displayed.
We use the Chern curvature convention
\begin{equation}\label{eq:chern-curvature}
 R^h_{i\bar j k\bar l}
 =-\partial_i\partial_{\bar j}h_{k\bar l}
  +h^{p\bar q}(\partial_i h_{k\bar q})
                   (\partial_{\bar j}h_{p\bar l}).
\end{equation}
Thus the holomorphic sectional curvature is defined by
\begin{equation}\label{eq:HSC}
 H_h(x,[\xi])
 =\frac{R^h(\xi,\bar\xi,\xi,\bar\xi)}{h(\xi,\bar\xi)^2},
 \qquad 0\neq\xi\in T^{1,0}_xM.
\end{equation}
We call $H_h$ quasi-negative if $H_h(x,[\xi])\leq0$ for every
$(x,[\xi])\in\mathbb P(T^{1,0}M)$ and there is a point $p\in M$ such
that $H_h(p,[\xi])<0$ for every $[\xi]\in\mathbb P(T^{1,0}_pM)$.
By compactness of the projective fiber over $p$ and continuity, this
implies that there is an open neighborhood $U$ of $p$ and a constant
$\eta>0$ such that $H_h(x,[\xi])\leq-\eta$ for every
$(x,[\xi])\in\mathbb P(T^{1,0}U)$.

For an auxiliary K\"ahler metric $g$, we set
\begin{equation}\label{eq:metric-conventions}
 \Delta_g f=g^{i\bar j}\partial_i\partial_{\bar j}f,
 \quad
 \dV_g=\frac{\omega_g^n}{n!},
 \quad
 \Ric(\omega_g)=-\ii\partial\dbar\log\det(g_{i\bar j}).
\end{equation}
We use the convention $dd^c=\ii\partial\dbar$.
In particular, $\Ric(\omega_g)$ represents $2\pi c_1(M)$,
and $\Delta_g f\leq0$ at a local maximum of a real function.

Set
\begin{equation}\label{eq:u-q}
 u=\tr_g h,
 \qquad q=\tr\bigl((g^{-1}h)^2\bigr).
\end{equation}
Here $g^{-1}h$ denotes the positive $g$-self-adjoint endomorphism
corresponding to $h$. If its eigenvalues are $\lambda_1,\ldots,\lambda_n$,
then
\begin{equation}\label{eq:eigenvalue-bound}
 u=\sum_i\lambda_i,\qquad
 q=\sum_i\lambda_i^2\geq\frac{u^2}{n}.
\end{equation}
Define the two curvature contractions
\begin{equation}\label{eq:Q-B}
 \begin{split}
 Q_g(h)&=g^{i\bar j}g^{k\bar l}R^h_{i\bar j k\bar l},
 \\[-2pt]
 B_g(h)&=g^{i\bar j}g^{k\bar l}R^h_{i\bar l k\bar j},
 \\[-2pt]
 \RR_g(h)&=\frac{Q_g(h)+B_g(h)}2.
 \end{split}
\end{equation}
Both are real-valued by the Hermitian symmetry of the Chern
curvature. For a general Hermitian metric, $Q_g(h)$ and $B_g(h)$
need not coincide.

The projective-space form of Berger's averaging trick \cite{Berger66} appears in
Yang--Zheng's proof of Proposition~1.3~\cite[Section~2]{YZ19}.
Proposition~1.3 itself concerns
the equivalence of the signs of holomorphic and real bisectional
curvature under a K\"ahler-like symmetry assumption; it is therefore
not the precise statement proved below.

\begin{lemma}\label{lem:averaging}
Fix a point of a complex manifold of dimension $n$, and let $g$
and $h$ be Hermitian metrics. If the Chern curvature tensor of $h$
at that point satisfies, for some $\kappa\geq0$,
\[
 R^h(\xi,\bar\xi,\xi,\bar\xi)
 \leq-\kappa\,h(\xi,\bar\xi)^2
 \quad\text{for every }\xi,
\]
then
\begin{equation}\label{eq:averaged-negativity}
 \RR_g(h)\leq-\frac{\kappa}{2}(u^2+q).
\end{equation}
In particular, $H_h\leq0$ implies $\RR_g(h)\leq0$, and strict
negativity in every direction at the point gives strict inequality
there.
\end{lemma}

\begin{proof}
Choose a $g$-unitary basis and write homogeneous coordinates on
$\PP^{n-1}$ as $[w^1:\cdots:w^n]$. Let $\omega_{\mathrm{FS}}$ be the
Fubini--Study form normalized by
$\int_{\PP^{n-1}}\omega_{\mathrm{FS}}^{n-1}=1$. In the form used by
Yang--Zheng~\cite[Section~2]{YZ19}, Berger's averaging trick
\cite{Berger66} gives
\begin{equation}\label{eq:projective-moment}
 \int_{\PP^{n-1}}
 \frac{w^i\overline{w^j}w^k\overline{w^l}}{|w|^4}
 \,\omega_{\mathrm{FS}}^{n-1}
 =\frac{\delta_{ij}\delta_{kl}+\delta_{il}\delta_{kj}}{n(n+1)}.
\end{equation}
The quotient is homogeneous of degree zero, so the integral is
well-defined on projective space. In a $g$-unitary frame, applying
\eqref{eq:projective-moment} to the curvature tensor gives
\begin{equation}\label{eq:average-curvature}
 \int_{\PP^{n-1}}
 \frac{R^h(w,\bar w,w,\bar w)}{|w|^4}\,\omega_{\mathrm{FS}}^{n-1}
 =\frac{Q_g(h)+B_g(h)}{n(n+1)}
 =\frac{2\RR_g(h)}{n(n+1)}.
\end{equation}
Applying the same identity to $h\otimes h$ gives
\begin{equation}\label{eq:average-metric}
 \int_{\PP^{n-1}}
 \frac{h(w,\bar w)^2}{|w|^4}\,\omega_{\mathrm{FS}}^{n-1}
 =\frac{u^2+q}{n(n+1)}.
\end{equation}
Integrating the assumed inequality and using
\eqref{eq:average-curvature} and \eqref{eq:average-metric} gives
\eqref{eq:averaged-negativity}, proving the lemma.
\end{proof}

\begin{remark}\label{rem:averaging-not-Royden}
The averaging is over the projective space of $g$-complex lines. Its
purpose is to control precisely the symmetrized contraction $\RR_g(h)$. The classical
Schwarz lemma and Royden's argument
\cite{YauSchwarz78,Royden80} motivate this calculation, but when $h$
is not K\"ahler one cannot simply replace $\RR_g(h)$ by $Q_g(h)$. The
next section accounts for the difference between these contractions.
\end{remark}

\section{An integrated Chern--Lu identity}
\label{sec:integrated}

In this section $g$ is K\"ahler and $h$ is an arbitrary Hermitian
metric on $M$. We assume $n\geq2$. The identity below is the
variable-curvature form of the calculation in
\cite[Lemma~3.2 and Appendix~A.1]{QT26}; we include a direct
derivation to isolate the terms caused by the failure of $h$ to be
K\"ahler.

\subsection{Connection difference and the exact term}

Let $\nabla^g$ and $\nabla^h$ be the Chern connections and put
\begin{equation}\label{eq:connection-difference}
 A^p{}_{ik}=\Gamma(h)^p{}_{ik}-\Gamma(g)^p{}_{ik},
 \qquad
 S^p{}_{ik}=\frac12\bigl(A^p{}_{ik}+A^p{}_{ki}\bigr).
\end{equation}
Here $\Gamma(h)^p{}_{ik}=h^{p\bar q}\partial_i h_{k\bar q}$ and
$\Gamma(g)^p{}_{ik}=g^{p\bar q}\partial_i g_{k\bar q}$ are the
Chern connection coefficients in local holomorphic coordinates.
Because $g$ is K\"ahler, the tensor
\begin{equation}\label{eq:torsion}
 T^p{}_{ik}=A^p{}_{ik}-A^p{}_{ki}
\end{equation}
is precisely the Chern torsion of $h$. We measure the two covariant
indices with $g$ and the output index with $h$. For example,
\begin{equation}\label{eq:mixed-norm}
 |A|_{g,h}^2
 =g^{i\bar j}g^{k\bar l}h_{p\bar q}
 A^p{}_{ik}\overline{A^q{}_{jl}}.
\end{equation}
The symmetric and skew-symmetric parts are orthogonal, so that
\begin{equation}\label{eq:norm-splitting}
 |A|_{g,h}^2=|S|_{g,h}^2+\frac14|T|_{g,h}^2.
\end{equation}

We package the second-derivative terms into the tensor
\begin{equation}\label{eq:D-tensor}
 \begin{split}
 D_{i\bar j k\bar l}={}&
 \partial_i\partial_{\bar j}h_{k\bar l}
 +\partial_k\partial_{\bar l}h_{i\bar j}-\partial_i\partial_{\bar l}h_{k\bar j}
 -\partial_k\partial_{\bar j}h_{i\bar l}.
 \end{split}
\end{equation}
Its tensoriality follows from the exterior-form identity (with
holomorphic differentials written before antiholomorphic ones):
\begin{equation}\label{eq:D-form}
 \ii\partial\dbar\omega_h
 =\sum_{\substack{i<k\\j<l}}
 D_{i\bar j k\bar l}\,
 \dd z^i\wedge\dd z^k\wedge\dd\bar z^j\wedge\dd\bar z^l.
\end{equation}
Write
\begin{equation}\label{eq:D-contraction}
 D_g=g^{i\bar j}g^{k\bar l}D_{i\bar j k\bar l}.
\end{equation}

\begin{lemma}[{\cite[Lemma~3.2, formulas~(3.5)--(3.6)]{QT26}}]\label{lem:stokes}
The contraction $D_g$ satisfies
\begin{equation}\label{eq:D-volume}
 D_g\,\dV_g
 =2\ii\partial\dbar\omega_h
       \wedge\frac{\omega_g^{n-2}}{(n-2)!}.
\end{equation}
Consequently,
\begin{equation}\label{eq:D-integral}
 \int_M D_g\,\dV_g=0.
\end{equation}
\end{lemma}

\begin{proof}
At the point under consideration, choose coordinates in which
$g_{i\bar j}=\delta_{ij}$. In the wedge product on the right of
\eqref{eq:D-volume}, a coefficient in \eqref{eq:D-form} contributes
only when the unordered pairs $\{i,k\}$ and $\{j,l\}$ agree. For
each $i<k$ it contributes $D_{i\bar i k\bar k}\dV_g$.
On the other hand, the contraction on the left is
\[
 D_g=\sum_{i,k}D_{i\bar i k\bar k}
     =2\sum_{i<k}D_{i\bar i k\bar k}.
\]
This proves the factor $2$ in \eqref{eq:D-volume}. Since $g$ is
K\"ahler, $\partial\omega_g=\dbar\omega_g=0$, and Stokes' theorem
gives
\[
 \int_M\ii\partial\dbar\omega_h\wedge\omega_g^{n-2}
 =\int_M\ii\partial\dbar
       \bigl(\omega_h\wedge\omega_g^{n-2}\bigr)=0.
\]
No closedness assumption on $\omega_h$ is needed.
\end{proof}

\subsection{Curvature symmetrization and torsion cancellation}

\begin{lemma}[{cf.~\cite[(2.11)]{Tang26b}}]\label{lem:curvature-difference}
The Chern curvature of $h$ satisfies
\begin{equation}\label{eq:double-antisymmetry}
 \begin{split}
 &R^h_{i\bar j k\bar l}+R^h_{k\bar l i\bar j}
  -R^h_{i\bar l k\bar j}-R^h_{k\bar j i\bar l}=h_{p\bar q}T^p{}_{ik}
                      \overline{T^q{}_{jl}}
                  -D_{i\bar j k\bar l}.
 \end{split}
\end{equation}
In particular,
\begin{equation}\label{eq:Q-symmetrized}
 Q_g(h)=\RR_g(h)+\frac14|T|_{g,h}^2-\frac14D_g.
\end{equation}
\end{lemma}

\begin{proof}
Substitute \eqref{eq:chern-curvature} into the left-hand side of
\eqref{eq:double-antisymmetry}. The second-derivative terms give
$-D_{i\bar j k\bar l}$. Write the quadratic term in the curvature as
\[
 h_{p\bar q}\Gamma(h)^p{}_{ik}
                   \overline{\Gamma(h)^q{}_{jl}}.
\]
The four quadratic terms combine to
\[
 h_{p\bar q}
 \bigl(\Gamma(h)^p{}_{ik}-\Gamma(h)^p{}_{ki}\bigr)
 \overline{\bigl(\Gamma(h)^q{}_{jl}-\Gamma(h)^q{}_{lj}\bigr)},
\]
which is the displayed torsion term.

Contracting \eqref{eq:double-antisymmetry} with
$g^{i\bar j}g^{k\bar l}$ gives
\[
 2\bigl(Q_g(h)-B_g(h)\bigr)=|T|_{g,h}^2-D_g.
\]
Together with the definition of $\RR_g(h)$, this is
\eqref{eq:Q-symmetrized}.
\end{proof}

To write the Ricci term, identify a real $(1,1)$-form with its
Hermitian coefficient matrix and set
\begin{equation}\label{eq:Ricci-contraction}
 I_g(h)=\tr\bigl(g^{-1}\Ric(\omega_g)g^{-1}h\bigr).
\end{equation}
Equivalently, in a $g$-unitary frame,
$I_g(h)=\sum_{i,k}\Ric(g)_{i\bar k}h_{k\bar i}$.

\begin{proposition}[{cf.~\cite[Lemma~3.2 and Appendix~A.1]{QT26}}]
\label{prop:identity}
For $u=\tr_g h$, one has the pointwise identity
\begin{equation}\label{eq:symmetrized-identity}
 \Delta_g u
 =|S|_{g,h}^2+I_g(h)-\RR_g(h)+\frac14D_g.
\end{equation}
Consequently,
\begin{equation}\label{eq:integrated-identity}
 \int_M\bigl(|S|_{g,h}^2+I_g(h)-\RR_g(h)\bigr)\,\dV_g=0.
\end{equation}
\end{proposition}

\begin{proof}
The Chern--Lu identity for the holomorphic identity map
$(M,g)\to(M,h)$ reads
\begin{equation}\label{eq:usual-Chern-Lu}
 \Delta_g u=|A|_{g,h}^2+I_g(h)-Q_g(h);
\end{equation}
see the Schwarz-lemma calculations in \cite{Lu68,YauSchwarz78}.
For completeness, choose $g$-normal holomorphic coordinates at a
point. Since $g$ is K\"ahler, $g_{i\bar j}=\delta_{ij}$ and
$\partial g_{i\bar j}=0$ at that point. Differentiating
$u=g^{k\bar l}h_{k\bar l}$ gives
\[
 \Delta_g u
 =\sum_{i,k,l}(\partial_i\partial_{\bar i}g^{k\bar l})
                     h_{k\bar l}
  +\sum_{i,k}\partial_i\partial_{\bar i}h_{k\bar k}.
\]
Using the K\"ahler curvature symmetries of $g$, the first term is
$I_g(h)$. Formula~\eqref{eq:chern-curvature} identifies the second
term with
\[
 -Q_g(h)+\sum_{i,k,p,q}h_{p\bar q}
           \Gamma(h)^p{}_{ik}\overline{\Gamma(h)^q{}_{ik}}
 =-Q_g(h)+|A|_{g,h}^2,
\]
since $\Gamma(g)=0$ at the chosen point. This proves
\eqref{eq:usual-Chern-Lu} directly, without using any K\"ahler
symmetry of $h$.

Substituting \eqref{eq:norm-splitting} and
\eqref{eq:Q-symmetrized} into \eqref{eq:usual-Chern-Lu} cancels
the two occurrences of $|T|_{g,h}^2/4$, and gives
\eqref{eq:symmetrized-identity}. Finally, since
$\int_M\Delta_g u\,\dV_g=0$, Lemma~\ref{lem:stokes} gives
\eqref{eq:integrated-identity}.
\end{proof}

\section{Positivity of the canonical bundle under quasi-negative curvature}
\label{sec:canonical}

We now prove Theorem~\ref{thm:main}, using the point $p$ and the
neighborhood $U$ provided by the quasi-negative hypothesis.
In complex dimension one,
\[
 \deg K_M=-\frac1{2\pi}\int_M H_h\,\omega_h>0,
\]
so the conclusion is immediate. We assume $n\geq2$ from now on.

\subsection{Nefness and localized strict negativity}

For a quasi-negative metric, define
\[
 \kappa(x)=-\max_{[\xi]\in\mathbb P(T^{1,0}_xM)}H_h(x,[\xi]).
\]
Then $\kappa\geq0$, and there is an open set $U\Subset M$ and a
constant $\kappa_0>0$ such that $\kappa\geq\kappa_0$ on $U$.
We first record the birational input. A line bundle is called nef
here if its first Chern class lies in the closure of the K\"ahler
cone.

\begin{lemma}[{\cite[Theorem~1.1(1)]{Tang26}}]\label{lem:nef}
Let $M$ be a compact K\"ahler manifold carrying a Hermitian metric
$h$ with $H_h\leq0$. Then $M$ contains no rational curves and
$K_M$ is nef.
\end{lemma}

\begin{proof}
We include the curve argument because the curvature metric need not
be K\"ahler. Suppose $f:\PP^1\to M$ is a nonconstant holomorphic
map. Equip $\PP^1$ with a Fubini--Study metric $\gamma$ normalized
by $\Ric(\omega_\gamma)=\lambda\omega_\gamma$, where $\lambda>0$,
and set $e=|\dd f|_{\gamma,h}^2$. For a local $\gamma$-unit vector $v$,
the Chern--Lu formula in
complex dimension one is
\[
 \Delta_\gamma e
 =|\nabla\dd f|^2+\lambda e
  -R^h\bigl(\dd f(v),\overline{\dd f(v)},
                      \dd f(v),\overline{\dd f(v)}\bigr).
\]
This is the usual holomorphic-map Bochner formula
\cite{Lu68,YauSchwarz78}; the target curvature term has only one
direction because the source is a curve. Since $H_h\leq0$, we
obtain $\Delta_\gamma e\geq\lambda e$. At a point where
$e$ attains its positive maximum, this contradicts
$\Delta_\gamma e\leq0$. Thus every such map is constant. The same
argument excludes singular rational curves after normalization.

For nefness, we use the reduction in
\cite[Theorem~1.1(1)]{Tang26}. By Ou's theorem
\cite[Theorem~1.1]{Ou25}, a compact K\"ahler manifold is uniruled
if and only if its canonical bundle is not pseudoeffective. Since
$M$ contains no rational curves, it is not uniruled, and $K_M$ is
pseudoeffective. If $K_M$ were not nef, the theorem of Cao and
H\"oring \cite[Theorem~1.3]{CH20} would produce a
$K_M$-negative rational curve. The pseudoeffectivity hypothesis in
lower dimensions required by that theorem is supplied by Ou's theorem,
which holds in every dimension. This contradicts the first
part of the proof, and establishes nefness.
\end{proof}

\subsection{A zero-direction curvature lemma}

The next lemma shows that nonpositive holomorphic sectional curvature
prevents the auxiliary forms from developing a null direction.

\begin{lemma}\label{lem:zero-direction}
Let $\gamma$ be a smooth, closed, semipositive $(1,1)$-form near a
point $x$, with $\operatorname{rank}\gamma_x=n-1$. If $H_h\leq0$,
then for every nonzero $v\in\ker\gamma_x$,
\[
 \operatorname{Ric}\bigl(\mu_h(\gamma)\bigr)(v,\bar v)\leq0,
 \qquad
 \mu_h(\gamma):=\frac{\omega_h\wedge\gamma^{n-1}}{(n-1)!}.
\]
Here, for a positive volume form $F\,dV_{\rm eucl}$, we set
$\operatorname{Ric}(F\,dV_{\rm eucl})=-dd^c\log F$.
In the coordinates used below,
$dV_{\rm eucl}:=\prod_{j=1}^n(\ii\,\dd z^j\wedge\dd\bar z^j)$.
Since the $n-1$ positive eigenvalues of $\gamma_x$ remain positive
near $x$, $\mu_h(\gamma)$ is a positive volume form on a smaller
neighborhood of $x$.
\end{lemma}

\begin{proof}
We work at the point $x$. Choose linear holomorphic coordinates so that
the kernel direction is the first coordinate direction and, at $x$,
\[
 v=\partial_1,
 \qquad
 (\gamma_{i\bar j})=\operatorname{diag}(0,1,\ldots,1),
 \qquad h_{1\bar a}=0\quad(a\geq2).
\]
The first two conditions normalize the kernel and the positive
restriction of $\gamma$ to its complement; the last one means that
the complement is chosen $h$-orthogonal to $v$.

Because $\gamma$ is semipositive and
$\gamma_{1\bar1}(x)=0$, the function $\gamma_{1\bar1}$ has a local
minimum at $x$. Hence
\[
 \partial_k\gamma_{1\bar1}(x)=0
 \qquad\text{for every }k.
\]
The closedness condition $d\gamma=0$ gives
\[
 \partial_k\gamma_{i\bar a}
 =\partial_i\gamma_{k\bar a}.
\]
Thus, for each $a\geq2$, the first derivatives
$\partial_k\gamma_{i\bar a}(x)$ are symmetric in $i$ and $k$. A
holomorphic coordinate change with identity linear part and suitable
quadratic terms changes these derivatives by a symmetric tensor.
Choosing that tensor appropriately, we may arrange that
\[
 \partial_k\gamma_{i\bar a}=0\quad(a\geq2)
\]
at $x$. This change does not affect the normalization at $x$ or its
first-order part. Applying closedness once more, together with
$\partial_k\gamma_{1\bar1}(x)=0$, gives
\[
 \partial_1\gamma_{i\bar j}(x)
 =\partial_{\bar1}\gamma_{i\bar j}(x)=0
 \qquad\text{for all }i,j.
\]

Write $\mu_h(\gamma)=F\,dV_{\rm eucl}$. Since
$F$ is the coefficient of $s$ in
$\det(\gamma+s h)$ at $s=0$, the normalization gives
$F(x)=h_{1\bar1}(x)$. Differentiating this determinant in the
$\partial_1$ direction, and using the vanishing of the first
derivatives above, gives
\[
\begin{aligned}
(\log F)_{1\bar1}
={}&(\log h_{1\bar1})_{1\bar1}
+\sum_{a=2}^n
  \partial_1\partial_{\bar1}\gamma_{a\bar a}
+\frac{\sum_{a=2}^n h_{a\bar a}}{h_{1\bar1}}\,
  \partial_1\partial_{\bar1}\gamma_{1\bar1}.
\end{aligned}
\]
The first term records the variation of $h$; the other two terms
record the variation of the positive and zero eigenvalues of
$\gamma$. Closedness also gives
\[
 \partial_1\partial_{\bar1}\gamma_{a\bar a}
 =\partial_a\partial_{\bar a}\gamma_{1\bar1}
 \qquad(a\geq2).
\]
Since $\gamma_{1\bar1}$ has a local minimum at $x$, both
$\partial_a\partial_{\bar a}\gamma_{1\bar1}$ and
$\partial_1\partial_{\bar1}\gamma_{1\bar1}$ are nonnegative there.
The coefficient multiplying the latter is positive because $h$ is
positive definite. Thus the last two terms are nonnegative.

It remains to examine $(\log h_{1\bar1})_{1\bar1}$. Direct
differentiation gives
\[
\begin{aligned}
(\log h_{1\bar1})_{1\bar1}
={}&-\frac{R^h_{1\bar1 1\bar1}}{h_{1\bar1}}+\frac{h^{p\bar q}(\partial_1h_{1\bar q})
 (\partial_{\bar1}h_{p\bar1})}{h_{1\bar1}}
 -\frac{|\partial_1h_{1\bar1}|^2}{h_{1\bar1}^2}.
\end{aligned}
\]
Let $\eta_q=\partial_1h_{1\bar q}$. Cauchy--Schwarz for the positive
Hermitian form $h^{-1}$ gives
\[
 h^{p\bar q}\eta_q\overline{\eta_p}
 \geq \frac{|\eta_1|^2}{h_{1\bar1}}.
\]
Hence the sum of the last two terms in the displayed formula is
nonnegative.
Finally, $H_h\leq0$ implies
$R^h_{1\bar1 1\bar1}\leq0$ in this direction, so the first term is
also nonnegative. We conclude that
$(\log F)_{1\bar1}\geq0$.

Since
$\operatorname{Ric}(F\,dV_{\rm eucl})=-dd^c\log F$ and
$v=\partial_1$ at $x$, this is exactly
\[
 \operatorname{Ric}\bigl(\mu_h(\gamma)\bigr)(v,\bar v)\leq0.
\]
\end{proof}

\subsection{The \texorpdfstring{$(n-1)$}{(n-1)}-Hessian auxiliary equation}

Set
\[
 \Omega=\frac{\omega_0^n}{n!},
 \qquad
 \alpha=2\pi c_1(K_M),
 \qquad
 \beta=[\omega_0].
\]
By Lemma~\ref{lem:nef}, $\alpha$ is nef, so $\alpha+t\beta$ is a
K\"ahler class for every $t>0$. We replace the usual twisted
K\"ahler--Einstein equation by
\begin{equation}\label{eq:hessian-equation}
 \omega_t=-\operatorname{Ric}(\Omega)+t\omega_0+dd^c\varphi_t>0,
 \qquad
 \mu_h(\omega_t)=e^{\varphi_t}\Omega.
\end{equation}
The equation is the case $m=n-1$ of the exponential complex Hessian
equation. We use the smooth solvability result
\cite[Theorem~9.2]{PSWZ25}; it gives an admissible solution, but not
automatically a positive definite one. The next lemma supplies the
missing positivity argument.

\begin{lemma}\label{lem:hessian}
For every $t>0$, equation~\eqref{eq:hessian-equation} has a smooth
solution, and the resulting form $\omega_t$ is K\"ahler with
$[\omega_t]=\alpha+t\beta$.
\end{lemma}

\begin{proof}
Fix $t>0$. We connect the desired equation to one with an obvious
positive solution. Let $\gamma_0=L\omega_0$. Since
\[
 \mu_h(L\omega_0)=L^{\,n-1}\mu_h(\omega_0),
\]
the Ricci form of $\mu_h(L\omega_0)$ does not depend on $L$. We may
therefore choose $L$ so large that
\[
 B_0:=\gamma_0+\operatorname{Ric}\bigl(\mu_h(\gamma_0)\bigr)>0.
\]
For $0\leq s\leq1$, set
\[
 B_s=(1-s)B_0+s\,t\omega_0,
 \qquad
 \chi_s=-\operatorname{Ric}(\Omega)+B_s.
\]
Both $B_0$ and $t\omega_0$ are positive closed forms, so $B_s>0$.
Furthermore,
\[
 [\chi_s]=(1-s)L\beta+s(\alpha+t\beta),
\]
because $[\operatorname{Ric}(\mu_h(\gamma_0))]=-\alpha$ and
$[-\operatorname{Ric}(\Omega)]=\alpha$. Hence $[\chi_s]$ is a
K\"ahler class for every $s$.
Choose a K\"ahler form $\widehat\omega_t\in\alpha+t\beta$ and define
\[
 \theta_s=(1-s)\gamma_0+s\widehat\omega_t.
\]
Then $\theta_s$ is K\"ahler and $[\theta_s]=[\chi_s]$. At $s=0$, define
\[
 \phi_0=\log\frac{\mu_h(\gamma_0)}{\Omega},
 \qquad
 \rho_0=-\phi_0.
\]
Since
\[
 dd^c\phi_0=\operatorname{Ric}(\Omega)
             -\operatorname{Ric}\bigl(\mu_h(\gamma_0)\bigr),
\]
we have $\chi_0=\theta_0+dd^c\rho_0$. At $s=1$, the
$\partial\bar\partial$-lemma gives a real function $\rho_1$ with
$\chi_1=\theta_1+dd^c\rho_1$. Set
\[
 \rho_s=(1-s)\rho_0+s\rho_1.
\]
Because $\chi_s-\theta_s$ depends affinely on $s$, this gives
$\chi_s=\theta_s+dd^c\rho_s$ for every $s$ and makes the family
smooth.
If $\gamma_s=\chi_s+dd^c\phi_s$ and
$\psi_s=\phi_s+\rho_s$, then $\gamma_s=\theta_s+dd^c\psi_s$ and
\[
 (\theta_s+dd^c\psi_s)^{n-1}\wedge\omega_h
 =e^{\psi_s+F_s}\omega_h^n,
 \qquad
 F_s=\log\frac{(n-1)!\,\Omega}{\omega_h^n}-\rho_s.
\]

Theorem~9.2 of~\cite{PSWZ25} applies because $\theta_s$ is K\"ahler
and hence $(\omega_h,n-1)$-positive. It gives a smooth admissible
solution $\psi_s$ for every $s$. The exponential term fixes the
additive constant, so the solution is unique. The linearized operator
has the form $\mathcal L_s-1$, where $\mathcal L_s$ is positive
elliptic. The maximum principle gives a trivial kernel. Since this
scalar elliptic operator has Fredholm index zero, it is invertible.
The implicit function theorem
therefore shows that the solutions vary smoothly with $s$; uniqueness
makes the local families agree on overlaps.

We next show that these admissible solutions are actually positive
definite. At $s=0$, the equation is solved by $\psi_0=0$; uniqueness
therefore identifies this with the solution chosen above. Moreover,
\[
 \chi_0+dd^c\phi_0=\gamma_0>0.
\]
For each $s$, set
\[
 \phi_s=\psi_s-\rho_s,
 \qquad
 \gamma_s=\chi_s+dd^c\phi_s.
\]
Thus $\gamma_s$ solves the original auxiliary equation.

Suppose that $\gamma_s$ first ceases to be positive definite at
$s=s_*$. By continuity, $\gamma_{s_*}\geq0$ and has a nonzero kernel.
The equation gives
$\mu_h(\gamma_{s_*})=e^{\phi_{s_*}}\Omega>0$. Since a semipositive
form with rank at most $n-2$ would make
$\omega_h\wedge\gamma_{s_*}^{\,n-1}$ vanish, the rank of
$\gamma_{s_*}$ is exactly $n-1$.
Taking Ricci forms in
$\mu_h(\gamma_s)=e^{\phi_s}\Omega$ gives
\[
 \operatorname{Ric}\bigl(\mu_h(\gamma_s)\bigr)
 =\operatorname{Ric}(\Omega)-dd^c\phi_s
 =-\gamma_s+B_s.
\]
For $0\ne v\in\ker\gamma_{s_*}$, this yields
\[
 \operatorname{Ric}\bigl(\mu_h(\gamma_{s_*})\bigr)(v,\bar v)
 =B_{s_*}(v,\bar v)>0,
\]
because $B_{s_*}$ is positive definite. This contradicts
Lemma~\ref{lem:zero-direction}, which gives the opposite inequality.
Thus $\gamma_s$ remains positive definite for all $s$. At $s=1$,
$B_1=t\omega_0$ and
$\chi_1=-\operatorname{Ric}(\Omega)+t\omega_0$, so
$\gamma_1=\omega_t$ solves~\eqref{eq:hessian-equation}.
\end{proof}

For the solution in Lemma~\ref{lem:hessian}, set
\[
 u_t=\operatorname{tr}_{\omega_t}\omega_h,
 \qquad
 dV_t=\frac{\omega_t^n}{n!},
 \qquad
 d\mu_t=u_t\,dV_t.
\]
The mixed-volume identity and equation~\eqref{eq:hessian-equation} give
\begin{equation}\label{eq:mixed-trace}
 d\mu_t=\frac{\omega_h\wedge\omega_t^{n-1}}{(n-1)!}
 =e^{\varphi_t}\Omega.
\end{equation}
In particular,
\begin{equation}\label{eq:ricci-hessian}
 \operatorname{Ric}(\omega_t)
 =-\omega_t+t\omega_0+dd^c\log u_t.
\end{equation}

\subsection{The integrated estimate}

We now combine the integrated Chern--Lu identity with the extra
Ricci term in~\eqref{eq:ricci-hessian}. For a K\"ahler metric $g$, put
\[
 u=\operatorname{tr}_g h,
 \qquad
 a_i=\partial_i u,
 \qquad
 b_i=g^{k\bar l}\nabla^g_k h_{i\bar l},
\]
and define
\[
 \sigma=\frac{a+b}{2},
 \qquad
 \tau=a-b.
\]
We now explain the relation between $\sigma$ and the symmetric
connection difference $S$. Fix a point and choose a $g$-unitary frame
which is normal for the Chern connection of $g$ at that point. Since
$\nabla^g g=0$ and $u=\sum_k h_{k\bar k}$ in this frame, we have
\[
 a_i=\sum_k\nabla^g_i h_{k\bar k}.
\]
Moreover, 
\[
 \nabla^g_i h_{k\bar l}=h_{p\bar l}A^p{}_{ik}.
\]
Consequently,
\[
 a_i=\sum_{k,p}h_{p\bar k}A^p{}_{ik},
 \qquad
 b_i=\sum_{k,p}h_{p\bar k}A^p{}_{ki}.
\]
Taking the average of these two identities gives the pointwise formula
\begin{equation}\label{eq:sigma-connection}
 \sigma_i=\sum_{k,p}h_{p\bar k}S^p{}_{ik}.
\end{equation}

It remains to estimate this contraction. For each fixed $i$, regard
$S_{ik}=(S^p{}_{ik})_p$ as a vector in $T^{1,0}M$ with the $h$-norm.
Then the right-hand side of~\eqref{eq:sigma-connection} is the sum of
the $h$-inner products $\langle S_{ik},e_k\rangle_h$, where
$e_1,\ldots,e_n$ is the chosen $g$-unitary frame. Cauchy--Schwarz in
the direct sum of these $n$ copies of $T^{1,0}M$ gives
\[
 |\sigma_i|^2
 \leq
 \left(\sum_k|S_{ik}|_h^2\right)
 \left(\sum_k|e_k|_h^2\right)
 =u\sum_k|S_{ik}|_h^2.
\]
Summing over $i$ yields
\begin{equation}\label{eq:sigma-bound}
 |\sigma|_g^2\leq u|S|_{g,h}^2.
\end{equation}
For a real function $f$, define
\[
 C_g(f)=\operatorname{tr}(g^{-1}dd^cf\,g^{-1}h).
\]
The following identity is just integration by parts, but we spell out
the indices. At a fixed point choose a $g$-unitary frame which is
normal for the Chern connection of $g$. Since $f$ is a scalar and
$g$ is the identity in this frame,
\[
 C_g(f)=\sum_{i,k}f_{i\bar k}h_{k\bar i},
 \qquad f_{i\bar k}=\partial_i\partial_{\bar k}f.
\]
Integrating first in the $\bar k$-direction and using that the volume
form of a K\"ahler metric is parallel, we obtain
\[
 \int_M C_g(f)\,dV_g
 =-\int_M\sum_{i,k}f_i\,\nabla^g_{\bar k}h_{k\bar i}\,dV_g.
\]
There is no boundary term because $M$ is compact. By the definition
of $b$,
\[
 \overline{b_i}=\sum_k\nabla^g_{\bar k}h_{k\bar i},
\]
so the last integrand is $-\sum_i f_i\overline{b_i}$. Since both
$C_g(f)$ and $f$ are real-valued, we may take the real part and write
the result invariantly as
\begin{equation}\label{eq:log-trace-integration}
 \int_M C_g(f)\,dV_g
 =-\int_M\operatorname{Re}\langle b,\partial f\rangle_g\,dV_g.
\end{equation}
Taking $f=\log u$ and using
$a=\sigma+\tau/2$, $b=\sigma-\tau/2$, we obtain
\begin{equation}\label{eq:log-trace-square}
 \int_M C_g(\log u)\,dV_g
 =\int_M\left(-\frac{|\sigma|_g^2}{u}
 +\frac{|\tau|_g^2}{4u}\right)dV_g.
\end{equation}

Apply the integrated identity~\eqref{eq:integrated-identity} with
$g=g_t$. From~\eqref{eq:ricci-hessian},
\[
 I_{g_t}(h)=-u_t+tP_t+C_{g_t}(\log u_t),
 \qquad
 P_t=\operatorname{tr}(g_t^{-1}\omega_0g_t^{-1}h)\geq0.
\]
Together with~\eqref{eq:log-trace-square}, this gives
\begin{equation}\label{eq:quasi-square-identity}
\begin{aligned}
0=\int_M\Bigg[
&|S|_{g_t,h}^2-\frac{|\sigma|_{g_t}^2}{u_t}
 +\frac{|\tau|_{g_t}^2}{4u_t}+tP_t-u_t-\RR_{g_t}(h)
\Bigg]dV_t.
\end{aligned}
\end{equation}
To read this identity, collect all terms except $-u_t-\RR_{g_t}(h)$
into
\[
 \mathcal E_t
 :=|S|_{g_t,h}^2-\frac{|\sigma|_{g_t}^2}{u_t}
   +\frac{|\tau|_{g_t}^2}{4u_t}+tP_t.
\]
The estimate $|\sigma|^2\leq u|S|^2$ shows that the first two terms
in $\mathcal E_t$ have a nonnegative sum, while the remaining two
terms are individually nonnegative. Hence $\mathcal E_t\geq0$ and
the identity implies
\begin{equation}\label{eq:integrated-curvature-bound}
 \int_M-\RR_{g_t}(h)\,dV_t
 =\int_Mu_t\,dV_t-\int_M\mathcal E_t\,dV_t
 \leq\int_Mu_t\,dV_t=J(t).
\end{equation}

We next estimate the left-hand side from below. At a point, choose a
$g_t$-unitary frame and let $\lambda_1,\ldots,\lambda_n$ be the
eigenvalues of $g_t^{-1}h$. Then
\[
 u_t=\sum_i\lambda_i,
 \qquad
 q_t=\sum_i\lambda_i^2\geq\frac{u_t^2}{n},
\]
where the last inequality is the elementary Cauchy--Schwarz inequality
for the $n$ numbers $\lambda_i$. By the averaging lemma and the
definition of $\kappa$,
\[
 -\RR_{g_t}(h)
 \geq\frac{\kappa}{2}(u_t^2+q_t)
 \geq\frac{n+1}{2n}\,\kappa u_t^2.
\]
Combining this pointwise estimate with~\eqref{eq:integrated-curvature-bound}
gives
\begin{equation}\label{eq:weighted-curvature-estimate}
 c_n\int_M\kappa u_t^2\,dV_t
 \leq\int_M u_t\,dV_t=J(t),
 \qquad c_n=\frac{n+1}{2n}.
\end{equation}

\subsection{Uniform mixed-volume control}

We now show that the open set $U$ on which the curvature is strictly
negative carries a fixed fraction of the mixed measure
$d\mu_t=u_t\,dV_t$. This is the step that converts the weighted
estimate~\eqref{eq:weighted-curvature-estimate} into an estimate for
the total mixed volume.

Normalize the potential by writing
\[
 v_t=\varphi_t-\sup_M\varphi_t.
\]
Then $v_t\leq0$ and $\sup_Mv_t=0$. Choose $C>0$ so that
\[
 \operatorname{Ric}(\Omega)-t\omega_0\geq-C\omega_0,
 \qquad 0<t\leq1.
\]
Equation~\eqref{eq:hessian-equation} then gives
$dd^cv_t>-C\omega_0$. Set $\Theta=C\omega_0$. Then
$\Theta+dd^cv_t\geq0$ and $\sup_Mv_t=0$, so Wu--Yau's
Lemma~3~\cite[Lemma~3]{WY16b} gives constants $\lambda,C_0>0$,
independent of $t$, such that
\[
 \int_Me^{-\lambda v_t}\,\Omega\leq C_0.
\]
This is stronger than the $L^1$ estimate needed here. Since $v_t\leq0$
and $x\leq\lambda^{-1}e^{\lambda x}$ for $x\geq0$, we obtain a uniform
bound for the negative part of $v_t$:
\begin{equation}\label{eq:normalized-L1}
 \int_M|v_t|\,\Omega\leq C_1,
\end{equation}
for $0<t\leq1$. Since $v_t\leq0$, this is the same as
$-\int_Mv_t\,\Omega\leq C_1$.

We next apply Jensen's inequality to the convex function $e^x$ on the
fixed set $U$. Since
$\int_Uv_t\,\Omega\geq-\int_M|v_t|\,\Omega\geq-C_1$, we obtain
\[
 \begin{aligned}
 \int_Ue^{v_t}\Omega
 &\geq \Omega(U)
 \exp\left(\frac1{\Omega(U)}\int_Uv_t\,\Omega\right)\geq \Omega(U)\exp\left(-\frac{C_1}{\Omega(U)}\right)=:b>0.
 \end{aligned}
\]
On the other hand, $v_t\leq0$ gives
$\int_Me^{v_t}\Omega\leq\Omega(M)$. Since
$d\mu_t=e^{\varphi_t}\Omega=e^{\sup_M\varphi_t}e^{v_t}\Omega$, the
factor $e^{\sup_M\varphi_t}$ cancels in the ratio, and hence
\begin{equation}\label{eq:mu-proportion}
 \frac{\mu_t(U)}{J(t)}
 =\frac{\int_Ue^{v_t}\Omega}{\int_Me^{v_t}\Omega}
 \geq\frac{b}{\Omega(M)}=:\delta>0.
\end{equation}
Thus $\mu_t(U)\geq\delta J(t)$, with $\delta$ independent of $t$.

We now use the strict negativity on $U$. Since $\kappa\geq\kappa_0$
there and $d\mu_t=u_t\,dV_t$, the weighted estimate gives
\[
 \int_Uu_t\,d\mu_t
 =\int_Uu_t^2\,dV_t
 \leq\frac1{\kappa_0}\int_U\kappa u_t^2\,dV_t
 \leq\frac{J(t)}{c_n\kappa_0}.
\]
Finally, apply Cauchy--Schwarz with respect to the measure $d\mu_t$:
\[
 \mu_t(U)^2
 \leq
 \left(\int_Uu_t\,d\mu_t\right)
 \left(\int_U\frac1{u_t}\,d\mu_t\right)
 \leq\frac{J(t)}{c_n\kappa_0}\,V(t),
 \qquad V(t):=\int_MdV_t,
\]
where we used
\[
 \int_U\frac1{u_t}\,d\mu_t
 =\int_UdV_t
 \leq V(t).
\]
Combining this with~\eqref{eq:mu-proportion} gives
\[
 \delta^2J(t)^2
 \leq\mu_t(U)^2
 \leq\frac{J(t)}{c_n\kappa_0}V(t).
\]
Since $J(t)>0$, division by $J(t)$ yields
\begin{equation}\label{eq:mixed-volume-control}
 J(t)=\int_M\operatorname{tr}_{\omega_t}\omega_h\,dV_t
 \leq C_2V(t),
 \qquad 0<t\leq1,
\end{equation}
where one may take
$C_2=(c_n\kappa_0\delta^2)^{-1}$, which is independent of $t$.

\subsection{Positive volume and ampleness}

We now use~\eqref{eq:mixed-volume-control} to obtain a positive
lower bound for $V(t)$ as $t\downarrow0$. Since $\omega_h$ and
$\omega_0$ are fixed positive definite forms on the compact
manifold $M$, there is a constant $a>0$ such that
$\omega_h\geq a\omega_0$. Taking the trace with respect to
$\omega_t$ gives
\[
 \operatorname{tr}_{\omega_t}\omega_0
 \leq a^{-1}\operatorname{tr}_{\omega_t}\omega_h
 =a^{-1}u_t.
\]

Because $[\omega_t]=\alpha+t\beta$, the total volume depends only
on this cohomology class:
\[
 V(t)=\frac1{n!}\int_M\omega_t^n
     =\frac1{n!}\int_M(\alpha+t\beta)^n.
\]
Differentiating the polynomial on the right, and then using the
trace identity for wedge products, we obtain
\[
 \begin{aligned}
 V'(t)
 &=\frac1{(n-1)!}\int_M(\alpha+t\beta)^{n-1}\wedge\beta\\
 &=\frac1{(n-1)!}\int_M\omega_t^{n-1}\wedge\omega_0\\
 &=\int_M\operatorname{tr}_{\omega_t}\omega_0\,dV_t.
 \end{aligned}
\]
The trace comparison above and~\eqref{eq:mixed-volume-control}
therefore give
\begin{equation}\label{eq:quasi-volume-derivative}
 V'(t)\leq a^{-1}J(t)
 \leq\frac{C_2}{a}V(t).
\end{equation}
Since $V(t)>0$ for $t>0$, we may divide by $V(t)$ and integrate:
\[
 \log V(1)-\log V(t)
 =\int_t^1\frac{V'(s)}{V(s)}\,ds
 \leq\frac{C_2}{a}(1-t).
\]
Thus
\[
 V(t)\geq V(1)\exp\left(-\frac{C_2}{a}(1-t)\right).
\]
The polynomial defining $V(t)$ is continuous at $t=0$, so letting
$t\downarrow0$ yields
\begin{equation}\label{eq:quasi-positive-volume}
 V(0)\geq V(1)\exp\left(-\frac{C_2}{a}\right)>0.
\end{equation}
Here $V(0)$ is defined by the cohomology polynomial; no convergence
of the metrics $\omega_t$ is needed.

Recall that $\alpha=2\pi c_1(K_M)$. Hence
\[
 \int_Mc_1(K_M)^n=\frac{n!}{(2\pi)^n}V(0)>0.
\]
The class $c_1(K_M)$ is nef by Lemma~\ref{lem:nef}.
Theorem~0.5 of Demailly--P\u{a}un~\cite{DP04} states that a nef
class with positive top self-intersection on a compact K\"ahler
manifold contains a K\"ahler current. It follows that $K_M$ is big.

\subsection{From positive volume to ampleness}

We now combine bigness with the absence of rational curves to complete
the proof of Theorem~\ref{thm:main}. We then record the
K\"ahler--Einstein consequence and the extension to Fujiki's class.

\begin{proof}[Proof of Theorem~\ref{thm:main}]
In dimension one, the degree computation at the beginning of
Section~\ref{sec:canonical} gives $\deg K_M>0$, so $K_M$ is ample.
Suppose now that $n\geq2$. The preceding argument shows that $K_M$
is big, hence $M$ is Moishezon. Since $M$ is also K\"ahler, it is
projective; see, for example,~\cite{Lazarsfeld04}.
Lemma~\ref{lem:nef} shows that $M$ contains no rational curves.
We may therefore apply~\cite[Lemma~2.1]{DT19}, which states that
the canonical bundle of a smooth projective variety is ample if it
is big and the variety contains no rational curves. Thus $K_M$ is ample.
\end{proof}

\begin{corollary}[{\cite{Aubin78,YauMA78}}]\label{cor:KE}
Under the assumptions of Theorem~\ref{thm:main}, $M$ admits a unique
K\"ahler--Einstein metric with K\"ahler form $\omega_{\rm KE}$ satisfying
\[
 \operatorname{Ric}(\omega_{\rm KE})=-\omega_{\rm KE},
 \qquad
 [\omega_{\rm KE}]=2\pi c_1(K_M).
\]
\end{corollary}

\begin{proof}
Since $K_M$ is ample by Theorem~\ref{thm:main}, this follows from the
Aubin--Yau theorem~\cite{Aubin78,YauMA78}.
\end{proof}

\begin{remark}\label{rem:scope}
The proof uses the K\"ahler structure of $M$ to construct the
auxiliary metrics, to integrate the exact term in
Lemma~\ref{lem:stokes}, and to pass from positive top self-intersection
to projectivity. The curvature metric $h$ itself need not be K\"ahler.
The argument gives no conclusion about the sign of the holomorphic
sectional curvature of the resulting K\"ahler--Einstein metric.
\end{remark}

\begin{remark}\label{rem:fujiki}
Theorem~\ref{thm:main} remains valid if $M$ is assumed to belong to
Fujiki's class $\mathcal C$, the class of compact complex manifolds
bimeromorphic to compact K\"ahler manifolds. Indeed, Qin and Tian
prove that a manifold in $\mathcal C$ carrying a Hermitian metric
with nonpositive Chern holomorphic sectional curvature is K\"ahler
\cite[Lemma~3.1]{QT26}. Since quasi-negativity implies $H_h\leq0$,
their result applies to $M$. Theorem~\ref{thm:main} then gives the
ampleness of $K_M$.
\end{remark}

\noindent\textbf{Acknowledgments.}
The author is sponsored by the National Key R\&D Program of China
(Grant No. 2024YFA1013200) and the National Natural Science Foundation
of China (Grant No. 12671100). The author used ChatGPT for proof exploration, assistance with drafting, and language polishing.  The author guided the research direction, checked the mathematical statements and proofs, and carried out the final revisions. The author takes full responsibility for the content of the paper.

\end{document}